\documentclass{amsart}
\usepackage{amsmath,amssymb,graphicx,hyperref,latexsym,url}

\newtheorem{theorem}{Theorem}
\newtheorem{lemma}[theorem]{Lemma}

\newtheorem*{Kummer theorem}{Kummer's theorem}
\newtheorem*{Kummer multinomial theorem}{Kummer's theorem for multinomial coefficients}

\newcommand{\N}{\mathbb{N}}
\newcommand{\Z}{\mathbb{Z}}
\newcommand{\bfc}{\mathbf{c}}
\newcommand{\bfd}{\mathbf{d}}
\newcommand{\bfm}{\mathbf{m}}
\DeclareMathOperator{\mult}{mult}
\DeclareMathOperator{\total}{total}
\newcommand{\colonequal}{\mathrel{\mathop:}=}

\begin{document}

\title{A matrix generalization of a theorem of Fine}
\author{Eric Rowland}
\address{
	Department of Mathematics \\
	Hofstra University \\
	Hempstead, NY \\
	USA
}
\dedicatory{To Jeff Shallit on his 60th birthday!}
\date{November 12, 2017}

\begin{abstract}
In 1947 Nathan Fine gave a beautiful product for the number of binomial coefficients $\binom{n}{m}$, for $m$ in the range $0 \leq m \leq n$, that are not divisible by $p$.
We give a matrix product that generalizes Fine's formula, simultaneously counting binomial coefficients with $p$-adic valuation $\alpha$ for each $\alpha \geq 0$.
For each $n$ this information is naturally encoded in a polynomial generating function, and the sequence of these polynomials is $p$-regular in the sense of Allouche and Shallit.
We also give a further generalization to multinomial coefficients.
\end{abstract}

\maketitle

\section{Binomial coefficients}\label{Binomial coefficients}

For a prime $p$ and an integer $n \geq 0$, let $F_p(n)$ be the number of integers $m$ in the range $0 \leq m \leq n$ such that $\binom{n}{m}$ is not divisible by $p$.
Let the standard base-$p$ representation of $n$ be $n_\ell \cdots n_1 n_0$.
Fine~\cite{Fine} showed that
\[
	F_p(n) = (n_0 + 1) \, (n_1 + 1) \, \cdots \, (n_\ell + 1).
\]
Equivalently,
\begin{equation}\label{commuted product}
	F_p(n) = \prod_{d = 0}^{p - 1} (d + 1)^{|n|_d},
\end{equation}
where $|n|_w$ is the number of occurrences of the word $w$ in the base-$p$ representation of $n$.
In the special case $p = 2$, Glaisher~\cite{Glaisher} was aware of this result nearly 50 years earlier.

Many authors have been interested in generalizing Fine's theorem to higher powers of $p$.
Since Equation~\eqref{commuted product} involves $|n|_d$, a common approach is to express the number of binomial coefficients satisfying some congruence property modulo~$p^\alpha$ in terms of $|n|_w$ for more general words $w$.
Howard~\cite{Howard}, Davis and Webb~\cite{Davis--Webb}, Webb~\cite{Webb}, and Huard, Spearman, and Williams~\cite{Huard--Spearman--Williams mod 9, Huard--Spearman--Williams mod p^2, Huard--Spearman--Williams mod 8} all produced results in this direction.
Implicit in the work of Barat and Grabner~\cite[\textsection 3]{Barat--Grabner} is that the number of binomial coefficients $\binom{n}{m}$ with $p$-adic valuation $\alpha$ is equal to $F_p(n) \cdot G_{p^\alpha}(n)$, where $G_{p^\alpha}(n)$ is some polynomial in the subword-counting functions $|n|_w$.
The present author~\cite{Rowland} gave an algorithm for computing a suitable polynomial $G_{p^\alpha}(n)$.
Spiegelhofer and Wallner~\cite{Spiegelhofer--Wallner} showed that $G_{p^\alpha}(n)$ is unique under some mild conditions and greatly sped up its computation by showing that its coefficients can be read off from certain power series.

These general results all use the following theorem of Kummer~\cite[pages~115--116]{Kummer}.
Let $\nu_p(n)$ denote the $p$-adic valuation of $n$, that is, the exponent of the highest power of $p$ dividing $n$.
Let $\sigma_p(m)$ be the sum of the standard base-$p$ digits of $m$.

\begin{Kummer theorem}
Let $p$ be a prime, and let $n$ and $m$ be integers with $0 \leq m \leq n$.
Then $\nu_p(\binom{n}{m})$ is the number of carries involved in adding $m$ to $n - m$ in base $p$.
Equivalently, $\nu_p(\binom{n}{m}) = \frac{\sigma_p(m) + \sigma_p(n - m) - \sigma_p(n)}{p - 1}$.
\end{Kummer theorem}

Kummer's theorem follows easily from Legendre's formula
\begin{equation}\label{Legendre}
	\nu_p(m!) = \frac{m - \sigma_p(m)}{p - 1}
\end{equation}
for the $p$-adic valuation of $m!$.

Our first theorem is a new generalization of Fine's theorem.
It provides a matrix product for the polynomial
\[
	T_p(n, x) \colonequal \sum_{m = 0}^n x^{\nu_p(\binom{n}{m})}
\]
whose coefficient of $x^\alpha$ is the number of binomial coefficients $\binom{n}{m}$ with $p$-adic valuation $\alpha$.
In particular, $T_p(n, 0) = F_p(n)$.
For example, the binomial coefficients $\binom{8}{m}$, for $m$ in the range $0 \leq m \leq 8$, are
\[
	1 \quad 8 \quad 28 \quad 56 \quad 70 \quad 56 \quad 28 \quad 8 \quad 1;
\]
their $2$-adic valuations are
\[
	0 \quad 3 \quad 2 \quad 3 \quad 1 \quad 3 \quad 2 \quad 3 \quad 0,
\]
so $T_2(8, x) = 4 x^3 + 2 x^2 + x + 2$.
The first few terms of the sequence $(T_2(n, x))_{n \geq 0}$ are as follows.
\[
	\begin{array}{r|r}
		n & \hfill T_2(n, x)\hfill  \\ \hline
		0 & 1 \\
		1 & 2 \\
		2 & x + 2 \\
		3 & 4 \\
		4 & 2 x^2 + \phantom{1} x + 2 \\
		5 & 2 x + 4 \\
		6 & x^2 + 2 x + 4 \\
		7 & 8
	\end{array}
	\qquad \qquad
	\begin{array}{r|r}
		n & \hfill T_2(n, x) \hfill \\ \hline
		8 & 4 x^3 + 2 x^2 + \phantom{1} x + 2 \\
		9 & 4 x^2 + 2 x + 4 \\
		10 & 2 x^3 + \phantom{1} x^2 + 4 x + 4 \\
		11 & 4 x + 8 \\
		12 & 2 x^3 + 5 x^2 + 2 x + 4 \\
		13 & 2 x^2 + 4 x + 8 \\
		14 & x^3 + 2 x^2 + 4 x + 8 \\
		15 & 16
	\end{array}
\]
The polynomial $T_p(n, x)$ was identified by Spiegelhofer and Wallner~\cite{Spiegelhofer--Wallner} as an important component in the efficient computation of the polynomial $G_{p^\alpha}(n)$.
Everett~\cite{Everett} was also essentially working with $T_p(n, x)$.

For each $d \in \{0, 1, \dots, p - 1\}$, let
\[
	M_p(d) =
	\begin{bmatrix}
		d + 1		& p - d - 1 \\
		d \, x		& (p - d) \, x
	\end{bmatrix}.
\]

\begin{theorem}\label{product}
Let $p$ be a prime, and let $n \geq 0$.
Let $n_\ell \cdots n_1 n_0$ be the standard base-$p$ representation of $n$.
Then
\[
	T_p(n, x) =
	\begin{bmatrix}
		1 & 0
	\end{bmatrix}
	M_p(n_0) \, M_p(n_1) \, \cdots \, M_p(n_\ell)
	\begin{bmatrix}
		1 \\
		0
	\end{bmatrix}.
\]
\end{theorem}

A sequence $s(n)_{n \geq 0}$, with entries in some field, is \emph{$p$-regular} if the vector space generated by the set of subsequences $\{ s(p^e n + i)_{n \geq 0} : \text{$e \geq 0$ and $0 \leq i \leq p^e - 1$} \}$ is finite-dimensional.
Allouche and Shallit~\cite{Allouche--Shallit} introduced regular sequences and showed that they have several desirable properties, making them a natural class.
The sequence $(F_p(n))_{n \geq 0}$ is included as an example of a $p$-regular sequence of integers in their original paper~\cite[Example~14]{Allouche--Shallit}.
It follows from Theorem~\ref{product} and \cite[Theorem~2.2]{Allouche--Shallit} that $(T_p(n, x))_{n \geq 0}$ is a $p$-regular sequence of polynomials.

Whereas Fine's product can be written as Equation~\eqref{commuted product}, Theorem~\ref{product} cannot be written in an analogous way, since the matrices $M_p(i)$ and $M_p(j)$ do not commute if $i \neq j$.

The proof of Theorem~\ref{product} uses Lemma~\ref{valuation relation}, which is stated and proved in general for multinomial coefficients in Section~\ref{Multinomial coefficients}.
The reason for including the following proof of Theorem~\ref{product} is that the outline is fairly simple.
The details relegated to Lemma~\ref{valuation relation} are not essentially simpler in the case of binomial coefficients, so we do not include a separate proof.

\begin{proof}[Proof of Theorem~\ref{product}]
For $n \geq 0$ and $d \in \{0, 1, \dots, p - 1\}$, let $m$ be an integer with $0 \leq m \leq p n + d$.
There are two cases.
If $(m \bmod p) \in \{0, 1, \dots, d\}$, then there is no carry from the $0$th position when adding $m$ to $p n + d - m$ in base $p$; therefore $\nu_p(\binom{p n + d}{m}) = \nu_p(\binom{n}{\lfloor m/p \rfloor})$ by Kummer's theorem.
Otherwise, there is a carry from the $0$th position, and $\nu_p(\binom{p n + d}{m}) = \nu_p(n) + \nu_p(\binom{n - 1}{\lfloor m/p \rfloor}) + 1$ by Lemma~\ref{valuation relation} with $i = 0$ and $j = 1$.
(Note that $n - 1 \geq 0$ here, since if $n = 0$ then $0 \leq m \leq d$ and we are in the first case.)
Since $\{0, 1, \dots, d\}$ has $d + 1$ elements and its complement has $p - d - 1$ elements, we have
\[
	\sum_{m = 0}^{p n + d} x^{\nu_p(\binom{p n + d}{m})}
	= (d + 1) \sum_{c = 0}^n x^{\nu_p(\binom{n}{c})}
	+ (p - d - 1) \sum_{c = 0}^{n - 1} x^{\nu_p(n) + \nu_p(\binom{n - 1}{c}) + 1}
\]
by comparing the coefficient of $x^\alpha$ on each side for each $\alpha \geq 0$.
Using the definition of $T_p(n, x)$, this equation can be written
\begin{equation}\label{recurrence 1}
	T_p(p n + d, x)
	= (d + 1) \, T_p(n, x)
	+ \begin{cases}
		0									& \text{if $n = 0$} \\
		(p - d - 1) \, x^{\nu_p(n) + 1} \, T_p(n - 1, x)	& \text{if $n \geq 1$}.
	\end{cases}
\end{equation}

Similarly, let $m$ be an integer with $0 \leq m \leq p n + d - 1$.
If $(m \bmod p) \in \{0, 1, \dots, d - 1\}$, then there is no carry from the $0$th position when adding $m$ to $p n + d - 1 - m$ in base $p$, and $\nu_p(p n + d) + \nu_p(\binom{p n + d - 1}{m}) = \nu_p(\binom{n}{\lfloor m/p \rfloor})$ by Lemma~\ref{valuation relation} with $i = 1$ and $j = 0$.
(Note that $p n + d - 1 \geq 0$ here, since if $n = d = 0$ there is no $m$ in the range $0 \leq m \leq p n + d - 1$.)
Otherwise there is a carry from the $0$th position, and $\nu_p(p n + d) + \nu_p(\binom{p n + d - 1}{m}) = \nu_p(n) + \nu_p(\binom{n - 1}{\lfloor m/p \rfloor}) + 1$ by Lemma~\ref{valuation relation} with $i = 1$ and $j = 1$.
Therefore
\[
	\sum_{m = 0}^{p n + d - 1} x^{\nu_p(p n + d) + \nu_p(\binom{p n + d - 1}{m})}
	= d \sum_{c = 0}^n x^{\nu_p(\binom{n}{c})}
	+ (p - d) \sum_{c = 0}^{n - 1} x^{\nu_p(n) + \nu_p(\binom{n - 1}{c}) + 1}.
\]
Multiplying both sides by $x$ and rewriting in terms of $T_p(n, x)$ gives
\begin{multline}\label{recurrence 2}
	\begin{cases}
		0									& \text{if $p n + d = 0$} \\
		x^{\nu_p(p n + d) + 1} \, T_p(p n + d - 1, x)		& \text{if $p n + d \geq 1$}
	\end{cases} \\
	= d \, x \, T_p(n, x)
	+ \begin{cases}
		0									& \text{if $n = 0$} \\
		(p - d) \, x \cdot x^{\nu_p(n) + 1} \, T_p(n - 1, x)	& \text{if $n \geq 1$}.
	\end{cases}
\end{multline}

We combine Equations~\eqref{recurrence 1} and \eqref{recurrence 2} into a matrix equation by defining
\[
	T'_p(n, x)
	\colonequal \begin{cases}
		0						& \text{if $n = 0$} \\
		x^{\nu_p(n) + 1} \, T_p(n - 1, x)	& \text{if $n \geq 1$}.
	\end{cases}
\]
For each $n \geq 0$, we therefore have the recurrence
\begin{equation}\label{matrix recurrence}
	\begin{bmatrix}
		T_p(p n + d, x) \\
		T'_p(p n + d, x)
	\end{bmatrix}
	=
	\begin{bmatrix}
		d + 1		& p - d - 1 \\
		d \, x		& (p - d) \, x
	\end{bmatrix}
	\begin{bmatrix}
		T_p(n, x) \\
		T'_p(n, x)
	\end{bmatrix},
\end{equation}
which expresses $T_p(p n + d, x)$ and $T'_p(p n + d, x)$ in terms of $T_p(n, x)$ and $T'_p(n, x)$.
The $2 \times 2$ coefficient matrix is $M_p(d)$.
We have
\[
	\begin{bmatrix}
		T_p(0, x) \\
		T'_p(0, x)
	\end{bmatrix}
	=
	\begin{bmatrix}
		1 \\
		0
	\end{bmatrix}
\]
for the vector of initial conditions, so the product
\[
	T_p(n, x) =
	\begin{bmatrix}
		1 & 0
	\end{bmatrix}
	M_p(n_0) \, M_p(n_1) \, \cdots \, M_p(n_\ell)
	\begin{bmatrix}
		1 \\
		0
	\end{bmatrix}
\]
now follows by writing $n$ in base $p$.
\end{proof}

We obtain Fine's theorem as a special case by setting $x = 0$.
The definition of $T'_p(n, x)$ implies $T'_p(n, 0) = 0$, so Equation~\eqref{matrix recurrence} becomes
\[
	\begin{bmatrix}
		F_p(p n + d) \\
		0
	\end{bmatrix}
	=
	\begin{bmatrix}
		d + 1		& p - d - 1 \\
		0		& 0
	\end{bmatrix}
	\begin{bmatrix}
		F_p(n) \\
		0
	\end{bmatrix},
\]
or simply
\[
	F_p(p n + d) = (d + 1) \, F_p(n).
\]

Equation~\eqref{recurrence 1} was previously proved by Spiegelhofer and Wallner~\cite[Equation~(2.2)]{Spiegelhofer--Wallner} using an infinite product and can also be obtained from an equation discovered by Carlitz~\cite{Carlitz}.
In fact Carlitz came close to discovering Theorem~\ref{product}.
He knew that the coefficients of $T_p(n, x)$ and $T'_p(n, x)$ can be written in terms of each other.
In his notation, let $\theta_\alpha(n)$ be the coefficient of $x^\alpha$ in $T_p(n, x)$, and let $\psi_{\alpha - 1}(n - 1)$ be the coefficient of $x^\alpha$ in $T'_p(n, x)$.
Carlitz gave the recurrence
\begin{align*}
	\theta_\alpha(p n + d) &= (d + 1) \theta_\alpha(n) + (p - d - 1) \psi_{\alpha - 1}(n - 1) \\
	\psi_\alpha(p n + d) &=
	\begin{cases}
		(d + 1) \theta_\alpha(n) + (p - d - 1) \psi_{\alpha - 1}(n - 1)	& \text{if $0 \leq d \leq p - 2$} \\
		p \psi_{\alpha - 1}(n)								& \text{if $d = p - 1$}.
	\end{cases}
\end{align*}
The first of these equations is equivalent to Equation~\eqref{recurrence 1}.
But to get a matrix product for $T_p(n, x)$, one needs an equation expressing $\psi_\alpha(p n + d - 1)$, not $\psi_\alpha(p n + d)$, in terms of $\theta$ and $\psi$.
That equation is
\[
	\psi_\alpha(p n + d - 1) = d \theta_\alpha(n) + (p - d) \psi_{\alpha - 1}(n - 1),
\]
which is equivalent to Equation~\eqref{recurrence 2}.
Therefore $\psi_\alpha(n - 1)$ (or, more precisely, $\psi_{\alpha - 1}(n - 1)$) seems to be more natural than Carlitz's $\psi_\alpha(n)$.

In addition to making use of $T_p(n, x)$, Spiegelhofer and Wallner~\cite{Spiegelhofer--Wallner} also utilized the normalized polynomial
\[
	\overline{T}_p(n, x) \colonequal \frac{1}{F_p(n)} T_p(n, x).
\]
It follows from Theorem~\ref{product} that the sequence $\left(\overline{T}_p(n, x)\right)_{n \geq 0}$ is also $p$-regular, since its terms can be computed using the normalized matrices $\frac{1}{d + 1} M_p(d)$.

We briefly investigate $T_p(n, x)$ evaluated at particular values of $x$.
We have already mentioned $T_p(n, 0) = F_p(n)$.
It is clear that $T_p(n, 1) = n + 1$.
When $p = 2$ and $x = -1$, we obtain a version of \href{http://oeis.org/A106407}{\textsf{A106407}}~\cite{OEIS}
with different signs.
Let $t(n)_{n \geq 0}$ be the Thue--Morse sequence, and let $S(n, x)$ be the $n$th Stern polynomial, defined by
\[
	S(n, x) =
	\begin{bmatrix}
		1 & 0
	\end{bmatrix}
	A(n_0) \, A(n_1) \, \cdots \, A(n_\ell)
	\begin{bmatrix}
		0 \\
		1
	\end{bmatrix},
\]
where
\[
	A(0) =
	\begin{bmatrix}
		x & 0 \\
		1 & 1
	\end{bmatrix}, \quad
	A(1) =
	\begin{bmatrix}
		1 & 1 \\
		0 & x
	\end{bmatrix},
\]
and as before $n_\ell \cdots n_1 n_0$ is the standard base-$2$ representation of $n$.

\begin{theorem}
For each $n \geq 0$, we have $T_2(n, -1) = (-1)^{t(n)} S(n + 1, -2)$.
\end{theorem}

\begin{proof}
Define the rank of a regular sequence to be the dimension of the corresponding vector space.
We bound the rank of $T_2(n, -1) - (-1)^{t(n)} S(n + 1, -2)$ using closure properties of $2$-regular sequences~\cite[Theorems~2.5 and 2.6]{Allouche--Shallit}.
Since the rank of $S(n, x)$ is $2$, the rank of $S(n + 1, -2)$ is at most $2$.
The rank of $(-1)^{t(n)}$ is $1$.
If two sequences have ranks $r_1$ and $r_2$, then their sum and product have ranks at most $r_1 + r_2$ and $r_1 r_2$.
Therefore $T_2(n, -1) - (-1)^{t(n)} S(n + 1, -2)$ has rank at most $4$, so to show that it is the $0$ sequence it suffices to check $4$ values of $n$.
\end{proof}

It would be interesting to know if there is a combinatorial interpretation of this identity.

\section{Multinomial coefficients}\label{Multinomial coefficients}

In this section we generalize Theorem~\ref{product} to multinomial coefficients.
For a $k$-tuple $\bfm = (m_1, m_2, \dots, m_k)$ of non-negative integers, define
\[
	\total \bfm \colonequal m_1 + m_2 + \cdots + m_k
\]
and
\[
	\mult \bfm \colonequal \frac{(\total \bfm)!}{m_1! \, m_2! \, \cdots \, m_k!}.
\]
Specifically, we count $k$-tuples $\bfm$ with a fixed total, according to the $p$-adic valuation $\nu_p(\mult \bfm)$.
The result is a matrix product as in Theorem~\ref{product}.
The matrices are $k \times k$ matrices with coefficients from the following sequence.

Let $c_{p, k}(n)$ be the number of $k$-tuples $\bfd \in \{0, 1, \dots, p-1\}^k$ with $\total \bfd = n$.
Note that $c_{p, k}(n) = 0$ for $n < 0$.
For example, let $p = 5$ and $k = 3$; the values of $c_{5, 3}(n)$ for $-k + 1 \leq n \leq p k - 1$ are
\[
	0 \quad 0 \quad 1 \quad 3 \quad 6 \quad 10 \quad 15 \quad 18 \quad 19 \quad 18 \quad 15 \quad 10 \quad 6 \quad 3 \quad 1 \quad 0 \quad 0.
\]
For $k \geq 1$, every tuple counted by $c_{p, k}(n)$ has a last entry $d$; removing that entry gives a $(k - 1)$-tuple with total $n - d$, so we have the recurrence
\[
	c_{p, k}(n) = \sum_{d = 0}^{p - 1} c_{p, k - 1}(n - d).
\]
Therefore $c_{p, k}(n)$ is an entry in the Pascal-like triangle generated by adding $p$ entries on the previous row.
For $p = 5$ this triangle begins as follows.
\[
\begin{array}{ccccccccccccccccc}
	1 \\
	1 & 1 & 1 & 1 & 1 \\
	1 & 2 & 3 & 4 & 5 & 4 & 3 & 2 & 1 \\
	1 & 3 & 6 & 10 & 15 & 18 & 19 & 18 & 15 & 10 & 6 & 3 & 1 \\
	1 & 4 & 10 & 20 & 35 & 52 & 68 & 80 & 85 & 80 & 68 & 52 & 35 & 20 & 10 & 4 & 1
\end{array}
\]
The entry $c_{p, k}(n)$ is also the coefficient of $x^n$ in $(1 + x + x^2 + \cdots + x^{p - 1})^k$.

For each $d \in \{0, 1, \dots, p - 1\}$, let $M_{p, k}(d)$ be the $k \times k$ matrix whose $(i, j)$ entry is $c_{p, k}(p \, (j - 1) + d - (i - 1)) \, x^{i - 1}$.
The matrices $M_{5, 3}(0), \dots, M_{5, 3}(4)$ are
\begin{align*}
	&\begin{bmatrix}
		1 & 18 & 6 \\
		0 & 15 x & 10 x \\
		0 & 10 x^2 & 15 x^2
	\end{bmatrix}, \quad
	\begin{bmatrix}
		3 & 19 & 3 \\
		x & 18 x & 6 x \\
		0 & 15 x^2 & 10 x^2
	\end{bmatrix}, \quad
	\begin{bmatrix}
		6 & 18 & 1 \\
		3 x & 19 x & 3 x \\
		x^2 & 18 x^2 & 6 x^2
	\end{bmatrix}, \\
	&\qquad \qquad
	\begin{bmatrix}
		10 & 15 & 0 \\
		6 x & 18 x & x \\
		3 x^2 & 19 x^2 & 3 x^2
	\end{bmatrix}, \quad
	\begin{bmatrix}
		15 & 10 & 0 \\
		10 x & 15 x & 0 \\
		6 x^2 & 18 x^2 & x^2
	\end{bmatrix}.
\end{align*}
For $k = 2$, the matrix $M_{p, 2}(d)$ is exactly the matrix $M_p(d)$ in Section~\ref{Binomial coefficients}.

We use $\N$ to denote the set of non-negative integers.
Let
\[
	T_{p, k}(n, x) = \sum_{\substack{\bfm \in \N^k \\ \total \bfm = n}} x^{\nu_p(\mult \bfm)}.
\]

\begin{theorem}\label{multinomial product}
Let $p$ be a prime, let $k \geq 1$, and let $n \geq 0$.
Let $e =
	\begin{bmatrix}
		1 & 0 & 0 & \cdots & 0
	\end{bmatrix}$
be the first standard basis vector in $\Z^k$.
Let $n_\ell \cdots n_1 n_0$ be the standard base-$p$ representation of $n$.
Then
\[
	T_{p, k}(n, x) =
	e \, M_{p, k}(n_0) \, M_{p, k}(n_1) \, \cdots \, M_{p, k}(n_\ell) \, e^\top.
\]
\end{theorem}

By setting $x = 0$ we recover a generalization of Fine's theorem for the number of multinomial coefficients not divisible by $p$; the top left entry of $M_{p, k}(d)$ is $c_{p, k}(d) = \binom{d + k - 1}{k - 1}$, so
\[
	T_{p, k}(n, 0) = \binom{n_0 + k - 1}{k - 1} \, \binom{n_1 + k - 1}{k - 1} \, \cdots \, \binom{n_\ell + k - 1}{k - 1}.
\]

The proof of Theorem~\ref{multinomial product} uses the following generalization of Kummer's theorem.
Recall that $\sigma_p(m)$ denotes the sum of the base-$p$ digits of $m$.
We write $\sigma_p(\bfm)$, $\lfloor\bfm/p\rfloor$, and $\bfm \bmod p$ for the tuples obtained by applying these functions termwise to the entries of $\bfm$.

\begin{Kummer multinomial theorem}
Let $p$ be a prime, and let $\bfm \in \N^k$ for some $k \geq 0$.
Then
\[
	\nu_p(\mult \bfm) = \frac{\total \sigma_p(\bfm) - \sigma_p(\total \bfm)}{p - 1}.
\]
\end{Kummer multinomial theorem}

This generalized version of Kummer's theorem also follows from Legendre's formula \eqref{Legendre}.
The following lemma gives the relationship between $\nu_p(\mult \bfm)$ and $\nu_p(\mult \lfloor\bfm/p\rfloor)$.

\begin{lemma}\label{valuation relation}
Let $p$ be a prime, $k \geq 1$, $n \geq 0$, $d \in \{0, 1, \dots, p - 1\}$, and $0 \leq i \leq k - 1$.
Let $\bfm \in \N^k$ with $\total \bfm = p n + d - i$.
Let $j = n - \total \lfloor\bfm/p\rfloor$.
Then $\total (\bfm \bmod p) = p j + d - i$, $0 \leq j \leq k - 1$, and
\[
	\nu_p\!\left(\frac{(p n + d)!}{(p n + d - i)!}\right) + \nu_p(\mult \bfm)
	= \nu_p\!\left(\frac{n!}{(n - j)!}\right) + \nu_p(\mult \lfloor\bfm/p\rfloor) + j.
\]
\end{lemma}

\begin{proof}
Let $\bfc = \lfloor\bfm/p\rfloor$ and $\bfd = (\bfm \bmod p) \in \{0, 1, \dots, p - 1\}^k$, so that $\bfm = p \, \bfc + \bfd$.
We have
\begin{align*}
	\total \sigma_p(\bfm) - \total \sigma_p(\bfc)
	&= \total \bfd \\
	&= \total \bfm - p \total \bfc \\
	&= p j + d - i \\
	&= p j + \sigma_p(p n + d) - \sigma_p(n) - i.
\end{align*}
In particular, $\total \bfd = p j + d - i$, as claimed; solving this equation for $j$ gives
\[
	j = \frac{-d + i + \total \bfd}{p},
\]
which implies the bounds
\[
	-1 + \frac{1}{p} = \frac{-(p - 1) + 0 + 0}{p} \leq j \leq \frac{0 + (k - 1) + (p - 1) k}{p} = k - \frac{1}{p}.
\]
Since $j$ is an integer, this implies $0 \leq j \leq k - 1$.

The generalized Kummer theorem gives
\begin{align*}
	(p - 1) & \left(\nu_p(\mult \bfm) - \nu_p(\mult \bfc)\right) \\
	&= (\total \sigma_p(\bfm) - \sigma_p(\total \bfm)) - (\total \sigma_p(\bfc) - \sigma_p(\total \bfc)) \\
	&= \total \sigma_p(\bfm) - \sigma_p(p n + d - i) - \total \sigma_p(\bfc) + \sigma_p(n - j).
\end{align*}
Since we established
\[
	\total \sigma_p(\bfm) - \total \sigma_p(\bfc)
	= p j + \sigma_p(p n + d) - \sigma_p(n) - i
\]
above, we can write
\begin{align*}
	(p - 1) & \left(\nu_p(\mult \bfm) - \nu_p(\mult \bfc)\right) \\
	&= p j + \sigma_p(p n + d) - \sigma_p(n) - i - \sigma_p(p n + d - i) + \sigma_p(n - j) \\
	&= (\sigma_p(p n + d) - i - \sigma_p(p n + d - i)) + (-\sigma_p(n) + j + \sigma_p(n - j)) + (p - 1) j \\
	&= (p - 1) \left(-\nu_p\!\left(\frac{(p n + d)!}{(p n + d - i)!}\right) + \nu_p\!\left(\frac{n!}{(n - j)!}\right) + j\right),
\end{align*}
where the last equality uses Legendre's formula.
Dividing by $p - 1$ and rearranging terms gives the desired equation.
\end{proof}

We are now ready to prove the main theorem of this section.

\begin{proof}[Proof of Theorem~\ref{multinomial product}]
Let $d \in \{0, 1, \dots, p - 1\}$, $0 \leq i \leq k - 1$, and $\alpha \geq 0$.
We claim that the map $\beta$ defined by
\[
	\beta(\bfm) \colonequal \left(\lfloor\bfm/p\rfloor, \bfm \bmod p\right)
\]
is a bijection from the set
\[
	A = \left\{\bfm \in \N^k : \text{$\total \bfm = p n + d - i$ and $\nu_p(\mult \bfm) =  \alpha - \nu_p\!\left(\frac{(p n + d)!}{(p n + d - i)!}\right)$}\right\}
\]
to the set
\begin{multline*}
	B = \bigcup_{j = 0}^{k - 1} \Bigg(\left\{\bfc \in \N^k : \text{$\total \bfc = n - j$ and $\nu_p(\mult \bfc) =  \alpha - \nu_p\!\left(\frac{n!}{(n - j)!}\right) - j$}\right\} \\
	\times \left\{\bfd \in \{0, 1, \dots, p - 1\}^k : \total \bfd = p j + d - i\right\}\Bigg).
\end{multline*}
Note that the $k$ sets in the union comprising $B$ are disjoint, since each tuple $\bfd$ occurs for at most one index $j$.
Lemma~\ref{valuation relation} implies that if $\bfm \in A$ then $\beta(\bfm) \in B$.

Clearly $\beta$ is injective, since $\beta(\bfm)$ preserves all the digits of the entries of $\bfm$.
It is also clear that $\beta$ is surjective, since a given pair $(\bfc, \bfd) \in B$ is the image of $p \, \bfc + \bfd \in A$.
Therefore $\beta : A \to B$ is a bijection.

Consider the polynomial
\[
	\sum_{\substack{\bfm \in \N^k \\ \total \bfm = p n + d - i}} x^{\nu_p\!\left(\frac{(p n + d)!}{(p n + d - i)!}\right) + \nu_p(\mult \bfm)}.
\]
The coefficient of $x^\alpha$ in this polynomial is $|A|$.
On the other hand, the coefficient of $x^\alpha$ in the polynomial
\[
	\sum_{j = 0}^{k-1} c_{p, k}(p j + d - i) \sum_{\substack{\bfc \in \N^k \\ \total \bfc = n - j}} x^{\nu_p\!\left(\frac{n!}{(n - j)!}\right) + \nu_p(\mult \bfc) + j}
\]
is $|B|$, since $c_{p, k}(p j + d - i)$ is the number of $k$-tuples $\bfd \in \{0, 1, \dots, p-1\}^k$ with $\total \bfd = p j + d - i$.
Since $A$ and $B$ are in bijection for each $\alpha \geq 0$, these two polynomials are equal.
Multiplying both polynomials by $x^i$ and rewriting in terms of $T_{p, k}(n, x)$ gives
\begin{multline*}
	\begin{cases}
		0															& \text{if $0 \leq p n + d \leq i - 1$} \\
		x^{\nu_p\!\left(\frac{(p n + d)!}{(p n + d - i)!}\right) + i} \, T_{p, k}(p n + d - i, x)	& \text{if $p n + d \geq i$}
	\end{cases} \\
	= \sum_{j = 0}^{k-1}
	\begin{cases}
		0																	& \text{if $0 \leq n \leq j - 1$} \\
		c_{p, k}(p j + d - i) \, x^i \cdot x^{\nu_p\!\left(\frac{n!}{(n - j)!}\right) + j} \, T_{p, k}(n - j, x)	& \text{if $n \geq j$}.
	\end{cases}
\end{multline*}
For each $i$ in the range $0 \leq i \leq k - 1$, define
\[
	T_{p, k, i}(n, x)
	\colonequal \begin{cases}
		0											& \text{if $0 \leq n \leq i - 1$} \\
		x^{\nu_p\!\left(\frac{n!}{(n - i)!}\right) + i} \, T_{p, k}(n - i, x)	& \text{if $n \geq i$}.
	\end{cases}
\]
Note that $T_{p, k, 0}(n, x) = T_{p, k}(n, x)$.
For $n \geq 0$, we therefore have
\[
	T_{p, k, i}(p n + d, x)
	= \sum_{j = 0}^{k-1} c_{p, k}(p j + d - i) \, x^i \, T_{p, k, j}(n, x).
\]
For each $i$, this equation gives a recurrence for $T_{p, k, i}(p n + d, x)$ in terms of $T_{p, k, j}(n, x)$ for $0 \leq j \leq k - 1$.
The coefficients of this recurrence are the entries of the matrix $M_{p, k}(d)$.
It follows from the definition of $T_{p, k, i}(n, x)$ that $T_{p, k, 0}(0, x) = 1$ and $T_{p, k, i}(0, x) = 0$ for $1 \leq i \leq k - 1$.
Therefore the vector of initial conditions is $
	\begin{bmatrix}
		1 & 0 & 0 & \cdots & 0
	\end{bmatrix}^\top$,
and the matrix product follows.
\end{proof}

A natural question suggested by this paper is whether various generalizations of binomial coefficients (Fibonomial coefficients, $q$-binomial coefficients, Carlitz binomial coefficients, coefficients of $(1 + x + x^2 + \cdots + x^a)^n$, other hypergeometric terms, etc.)\ and multinomial coefficients have results that are analogous to Theorems~\ref{product} and \ref{multinomial product}.

\end{document}